\documentclass[12pt]{amsart}
\usepackage{amssymb,amsmath,amsfonts}
\usepackage{mathrsfs}
\usepackage{ae,aecompl}
\usepackage[T1]{fontenc}
\usepackage[pdftex]{graphicx}
\usepackage{enumerate}
\usepackage{mathrsfs}
\usepackage[usenames]{color}
\usepackage{epstopdf}

\renewcommand{\mathcal}{\mathscr}
\renewcommand{\epsilon}{\varepsilon}

\newcommand{\dimH}{\dim_H}
\newcommand{\diam}{\textrm{diam}}

\newcommand{\RR}{\mathbb R}
\newcommand{\EE}{\mathbb E}
\newcommand{\PP}{\mathbb P}
\newcommand{\NN}{\mathbb N}

\DeclareMathOperator{\dist}{dist}

\newtheorem{theorem}{Theorem}[section]
\newtheorem{lemma}[theorem]{Lemma}
\newtheorem{corollary}[theorem]{Corollary}

\theoremstyle{definition}

\newtheorem{example}[theorem]{Example}

\theoremstyle{remark}
\newtheorem{remark}[theorem]{Remark}

\numberwithin{equation}{section}

\begin{document}

\title{Distribution of  random Cantor sets  on Tubes }
\author{Changhao Chen}
\address{Department of Mathematical Sciences, P.O. Box 3000, 90014
University of Oulu, Finland}
\email{changhao.chen@oulu.fi}

\date{\today}

\thanks{The author was supported by the Vilho, Yrj{\"o}, and Kalle V{\"a}is{\"a}l{\"a} foundation.}

\begin{abstract}
We show that there exist $(d-1)$ - Ahlfors regular  compact sets $E \subset \mathbb{R}^{d}, d\geq 2$ such that for any $t< d-1$, we have  
\[
\sup_T \frac{\mathcal{H}^{d-1}(E\cap T)}{w(T)^t}<\infty
\]
where the supremum is over all tubes $T$ with width $w(T) >0$. This settles a question of T. Orponen. The sets  we construct are random Cantor sets, and the method combines  geometric and probabilistic estimates on the intersections of these random Cantor sets with affine subspaces. 
\end{abstract}

\maketitle

\section{introduction}

A set $E \subset \RR^d (d \geq 2)$ is called tube null if for any $\epsilon> 0$, there exist countable many tubes $\{T_i\}$  covering $E$ and $\sum_i w(T_i)^{d-1} < \epsilon$. Here and in what follows, a tube $T$ with width $w= w(T) > 0$ is the $w/2$- neighborhood of some line in $\mathbb{R}^d$. We always assume that our tubes have positive width. 

This notion comes from the study of the localisation problem of the Fourier transform in dimension $d\geq 2$ ( this problem can be regarded as looking for the analogues of Riemann's localization principle in higher dimensions).  In \cite{csv}, they proved that if  $E \subset B$ ( here $B$ is the unit ball of $\RR^d$) is tube null, then $E$ is a Set of Divergence for the Localisation Problem ($SDLP$). It's an open problem whether every  $SDLP$ is  tube null, for more details see \cite{csv}. 


It's easy to see that a set $E \subset \RR^d$ with $\mathcal{H}^{d-1}(E)=0$ is tube null. Indeed,   \cite[Proposition 7]{csv} claims that if  $E \subset \RR^d$ with $0<\mathcal{H}^{d-1}(E)<\infty$, then $E$ is tube null. This implies 
\begin{equation}\label{start}
\sup_T \frac{\mathcal{H}^{d-1}(E\cap T)}{w(T)^{d-1}}=\infty. 
\end{equation} 
Since if there is a positive constant $C$ such that $\mathcal{H}^{d-1}(E\cap T) \leq C w(T)^{d-1}$ for all tubes $T$, then for any countable family of tubes $\{T_i\}$  which cover $E$, we have 
\[
\sum_i w(T_i)^{d-1} \geq C^{-1}\sum_i \mathcal{H}^{d-1}(E\cap T_i) \geq C^{-1}\mathcal{H}^{d-1}(E),
\]
which would contradict the tube nullity of $E$. Thus \eqref{start} holds.  In \cite{h}, they showed that the Von Koch curve is tube null. For more tube null examples, see \cite{csv}.

For the sets which are not tube null, in \cite{csv} they showed that for any $s \in (d-1/2, d)$, there exists set $E$ with $\dim_H(E)=s$ and $E$ is not tube null. The sharp low bound of above $s$ was obtained in \cite{ss}, they proved that there exist set with Hausdorff dimension $d-1$ which are not tube null (thus answered the question of \cite{csv}).  

Motivated by \cite[Proposition 1 ]{c},  Carbery asks to determine which pairs $(s, t) \in [0, d] \times [0, d]$ are
admissible in the sense that there exists a set $E \subset \mathbb{R}^d$ with $0<\mathcal{H}^{s}(E)<\infty$ and satisfies
\begin{equation}\label{admissible}
\sup_T \frac{\mathcal{H}^{s}(E\cap T)}{w(T)^t}<\infty.
\end{equation}
This problem can be regarded as to concern the distribution of sets on tubes. By the works of \cite{c, csv, ss, t} (different contributions), we know that all the pairs $(s, t)$ with $t \leq \min \{ d- 1, s \}$ except $(d-1, d-1)$ are admissible. In \cite{t}, Orponen raised the following question: is it possible to construct a set $E\subset \mathbb{R}^d$ with $0 < \mathcal{H}^{d-1} (E) < \infty$ such that for every $t < d-1$, 
\begin{equation}\label{question}
\sup_T \frac{\mathcal{H}^{d-1}(E\cap T)}{w(T)^t}<\infty ?
\end{equation}

We are able to settle this question.

\begin{theorem}\label{answer}
There exists a $(d-1)$ - Ahlfors regular  compact set $E \subset \mathbb{R}^{d}$, such that for every $t< d-1$,  
\begin{equation}
\sup_T \frac{\mathcal{H}^{d-1}(E\cap T)}{w(T)^t}<\infty.
\end{equation}
Recall that $E \subset \mathbb{R}^d$ is called $Q$-Ahlfors regular for $0<Q \leq d$, if  there exist positive constant $C$ such that $r^Q/C \leq \mathcal{H}^{Q}(E\cap B(x,r)) \leq C r^Q$
for all  $x \in E$ and $0<r< \diam(E),$ where $\diam(E)$ denotes the diameter of $E$. 
\end{theorem}

The paper is organised as follows. The random Cantor sets are introduced in Section 2 together with the required notations, definitions and results.
In Section 3 we present some geometric lemmas. Section 4 contains the main probabilistic argument. The last Section contains further discussion concerning our model and some concrete examples.

\bigskip
\noindent\textbf{Acknowledgements.} I am grateful to my  supervisor Ville Suomala for his guidance about the question in \cite{t} and for sharing his ideas. I also would like to thank
the anonymous referee for carefully reading the manuscript and giving helpful comments.

\section{Random Cantor sets and their projections}\label{random Cantor sets}

In this section, we define the random Cantor sets  and state our results for them.  Closely related random models have been consider in \cite{ckls} and \cite {ss}.

Let $(M_n)$ and $(N_n)$ be sequences of integers with $1\leq N_n\leq M_n^d$ ($M_n\ge 2$) for all $n$. Denote $r_n=\prod_{k=1}^n M_k^{-1}$, and $P_n=\prod_{k=1}^n N_k$. We decompose the unit cube $[0,1]^d$ into $M_1^d$ interior disjoint $M_1$-adic closed subcubes and randomly choose interior disjoint $N_1\leq M_1^d$ of these closed subcubes such that each of the closed subcubes has the same probability (i.e.$ N_1/ M_1^{d}$) of being chosen, and denote their union by $E_1$. Given $E_{n}$, a random collection of $P_n$  interior disjoint $r_n$ - adic closed subcubes of $[0,1]^{d},$  independently inside each of these closed cubes we choose $N_{n+1}$ interior disjoint  ($r_{n+1}$) -adic closed subcubes such that each of these closed subcubes has the same probability (i.e.$ N_{n+1}/ M_{n+1}^{d}$) of being chosen. Let $E_{n+1}$ be the union of the chosen closed cubes.  Denote by $\omega$ the element in the probability space $\Omega$ induced by the construction described above. Let $E=E^\omega$ be the random limit set
\[
E=\bigcap_{n=1}^\infty E_n.
\]
We also denote the random limit set by $E(M_n, N_n)$  when we want to stress the connection to the deterministic sequences $M_n$ and  $N_n$.

\begin{remark}\label{con}
One natural way to choose subcubes is that we first randomly choose one  such that every subcube has the same probability of being choosen. Then we choose the second subcube from the remaining subcubes such that every subcubes has the same probability of being  chosen, and go on this way. But in fact, the  above model contains more general random Cantor sets. For two specific examples see Example \ref{example}.
\end{remark}

\noindent\textbf{Important assumption:}  In this paper, we assume that $M_n$ is uniformly bounded which means that  there exists $M \in \NN$, such that  $M_n \leq M$ for every $n\in \NN$.  Then it's easy to see that all the Cantor sets $E(M_n, N_n)$ have Hausdorff dimension $s$, where \begin{equation}\label{dimension}
s=\liminf _{n\to\infty} \frac{\log P_n}{-\log r_n}.
\end{equation}

Let $G(d,m)$ denote the family of all $m$-dimensional linear subspaces of $\mathbb{R}^d$ and $A(d, m )$ denote the family of all $m$- dimensional planes of $\mathbb{R}^{d}$ that intersect the cube $[0,1]^d$. For every $V\in A(d,m)$, denote by $\pi_V$  the orthogonal projection onto $V$ and by $\dimH F$ the Hausdorff dimension of a set $F$. Recall the classical Marstand- Mattila projection theorem (See e.g \cite{fa}, \cite{ma}): Let $F \subset \mathbb{R}^d$ ($d\geq 2$) be a Borel set with Hausdorff dimension $s$. If $s\leq k$, then the orthogonal projection of $F$ onto almost all $k$-planes has Hausdorff dimension $s$;  if $s> k$, then the orthogonal projection of $F$ onto almost all $k$-planes has positive $k$-dimensional Lebesgue measure. 

Recently, there has been a growing interest in showing that for various random fractals there are a.s. no exceptional directions in the projection theorem. We will prove the following projection theorem for the above random Cantor sets.

\begin{theorem} \label{projection of sets} 
If $s\leq k$, then almost surely $\dimH \pi_V (E)= s$ for all $ V \in G(d,k)$.
\end{theorem}

For other random sets, same kind of results have been recently obtained e.g. in \cite{ckls, rs1, rs2, ss, ss2,  sv}. For $V \in G(d,k)$ such that $V^{\perp}$ is not parallel to any coordinate hyperplane, the claim of Theorem \ref{projection of sets} follows from \cite[Theorem 10.1]{ss2}. In this paper, we give a direct proof for Theorem \ref{projection of sets} without relying on the theory of general spatially independent martingales developed in \cite{ss2}. In particular, we verity in detail the claim of \cite[Remark 10.3 (ii)]{ss2} for the model at hand.

We consider the natural random measure on the random Cantor set. We denote by $D_n=D_n([0,1]^{d}), n\in \NN$ all the $r_n$- adic closed subcubes  of  the unit cube $[0,1]^{d}$. Let $E=\bigcap_{n=1}^\infty E_n$ be a realization. For any $n$ and $Q \in D_n$, define 
\begin{equation*}
\mu_0(Q)=
  \begin{cases}
    P_n^{-1}&\text{ if } Q \subset E_n\\
    0&\text{ otherwise } \,.
  \end{cases}
\end{equation*}
By Kolmogorov's extension theorem, there is a unique measure $\mu$ on $[0,1]^d$ such that 
$\mu(Q)=\mu_0(Q)$ for any $Q \in D_n, n\in \NN.$

In the following, tubular neighbourhoods of the elements in $A(d,m)$ are called strips ($1\leq m\leq d-1$). More precisely, a strip $S$ of width $w(S)=\delta>0$, defined by an element  $W \in A(d,m)$, is the set
\[
S=\{x \in \mathbb{R}^d\mid\dist(x,W)<\delta/2\}
\]
where $\dist$ is the Euclidean distance. We also denote this strip by $S(W)$ when it was induced by $W$. Denote by $S(d, m)$  all the strips induced by the element of $A(d,m)$ as above. Notice that we call the strips in $S(d, 1)$ tubes. 

Theorem \ref{projection of sets} is easily deduced from the following estimate for the projections of the  measure $\mu$.

\begin{lemma} \label{main lemma}
If $s \leq k$, then almost surely  for any $0<t<s$,  
\begin{equation}\label{main equation}
\sup_{S \in S(d, d-k)} \frac{\mu \left(E\cap S\right)}{w\left(S\right)^{t}} <\infty.
\end{equation}
\end{lemma}

Lemma \ref{main lemma} will be proved in Section \ref{probability part}. Next we  prove Theorem \ref{projection of sets} and Theorem \ref{answer} assuming that Lemma \ref{main lemma} holds.

\begin{proof}[Proof of Theorem \ref{projection of sets}] Clearly $\dimH \pi_V(E)\le \dimH(E)\le s$ for all $V \in G(d,k)$, so it remains to verify the lower bound.

Using Lemma \ref{main lemma} we see that, almost surely, the estimate 
\[
\frac{(\pi_V)_*\mu(B(x,r))}{(2r)^{t}}=\frac{\mu \left(E\cap S'\right)}{(2r)^{t}}  \leq\sup_{S \in S(d, d-k)} \frac{\mu \left(E\cap S\right)}{w\left(S\right)^{t}}<\infty
\]
holds for all $V \in G(d,k)$, $x \in V$ and $r$ and simultaneously for all $t < s$, where $(\pi_V)_*\mu$ is the image measure of $\mu$ under the orthogonal projection of $\pi_V$ and $S'$ is the strip with width $2r$  induced by orthogonal complement of $V$ at the point $x$. Thus with  full probability $\dimH \pi_V(E)\geq t$ holds for all $V \in G(d,k)$ (See e.g. \cite[Chapter 4]{fa}. )
Approaching $s$ along a sequence gives, almost surely for
all $V \in G(d,k)$, the lower bound $\dim_H E\geq s$.
\end{proof}

We prove Theorem \ref{answer} by choosing  $M_n=2$ and $N_n= 2^{d-1}$ for all $n \in \NN$ in the above random construction.

\begin{proof}[Proof of Theorem \ref{answer}]
Let $M_n=2$ and  $N_n= 2^{d-1}$ for all $n \in \NN$. Then for every $E \in E(2,2^{d-1})$ and for the natural measure  $\mu$ on $E$, we have that 
\begin{equation}\label{compare}
\mu(B(x,r)) \asymp r^{d-1} 
\end{equation}
for $x \in E$ and $0<r<1$ where the symbol $\asymp$ means that the ratio of both sides is bounded above and below by positive and finite constants which does not depend on $x$ and $r$.
Thus we have that
$\mu \asymp \mathcal{H}^{d-1}|_E$ (See e.g.\cite[Chapter 6]{ma}), and so we can replace $\mu$ by $\mathcal{H}^{d-1}$ in \eqref{main equation}. It implies that almost 
surely for any $t <d-1$, we have
\begin{equation}\label{lastt}
\sup_{S \in S(d, d-k)} \frac{\mathcal{H}^{d-1} \left(E\cap S\right)}{w\left(S\right)^{t}} <\infty.
\end{equation}
Since $\mu$ is a probability measure, it follows  that $0<\mathcal{H}^{d-1}(E)<\infty$. 
By \eqref {compare} all the sets $E(2, 2^{d-1})$ are $(d-1)$-Alhfors regular. Thus we complete the proof.
\end{proof} 

Now we present two concrete examples of random Cantor sets on $\RR^2$ that fit into our general frame work. 

\begin{example}\label{example}
Consider the unit cube $[0,1]^2$. Let $N_n=M_n =2$ for all $n\in \NN$. 
Let $Q_1 = [0,1/2]\times [0,1/2], Q_2 = [1/2,1]\times [0,1/2], Q_3 = [1/2,1]\times [1/2,1]$ and $Q_4= [0,1/2]\times[1/2,1]$. Let  $L = \{Q_1, Q_4\}, R = \{Q_2, Q_3\}, \widetilde{L} = \{Q_1, Q_3\}$, and $\widetilde{D} = \{Q_2, Q_4\}$ corresponding to 'left', 'right', 'bottom left and top right', and 'bottom right and top left' subcubes of the unit cube.  

Let $E_1 = L$ or $R$ 
with the same probability $1/2$. Note that then every subcube has the same probability $1/2$ of being chosen. Given 
$E_n$, a random collection of $2^n$ interior disjoint $2^n$-adic closed subcubes of $[0,1]^2$, independently inside each of these cubes we choose the 'left' or 'right' column of the subcubes in the same way as $E_1 \subset [0,1]^2$. Let $E_{n+1}$
be the union of the chosen cubes. In the end we have the limit set (for an example see Figure \ref{fig3})
\[
E= \bigcap^{\infty}_{n=1} E_n.
\]


\begin{figure}
\centering 
\includegraphics[width=0.8\textwidth]{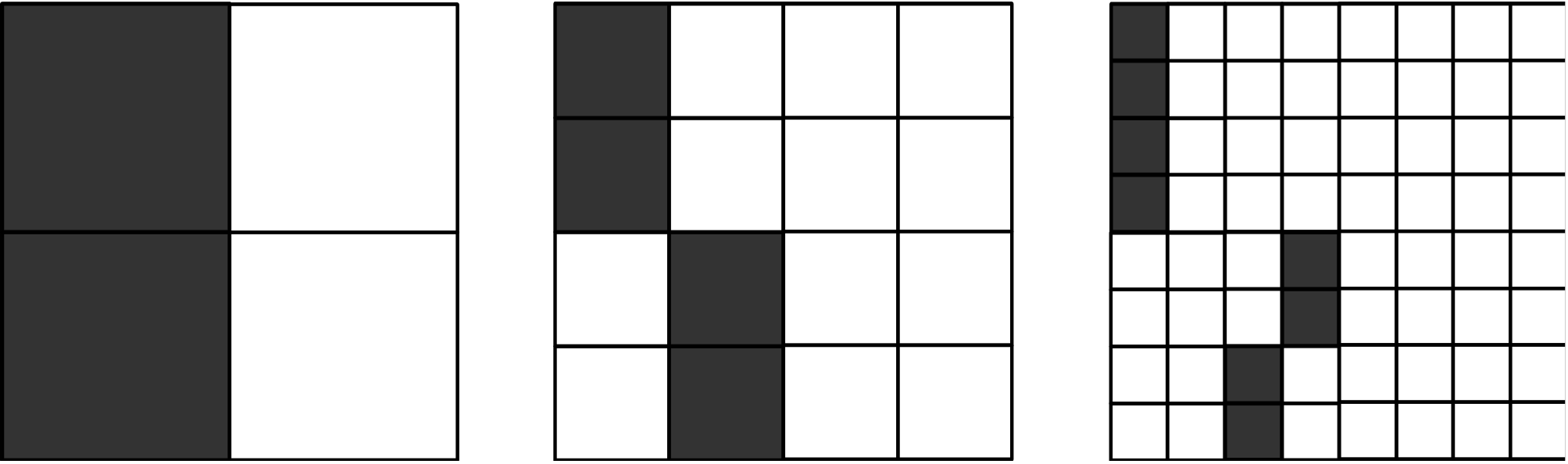}
\caption{The first three steps in the construction of E}
\label{fig3}
\end{figure}

\begin{figure}
\centering 
\includegraphics[width=0.8\textwidth]{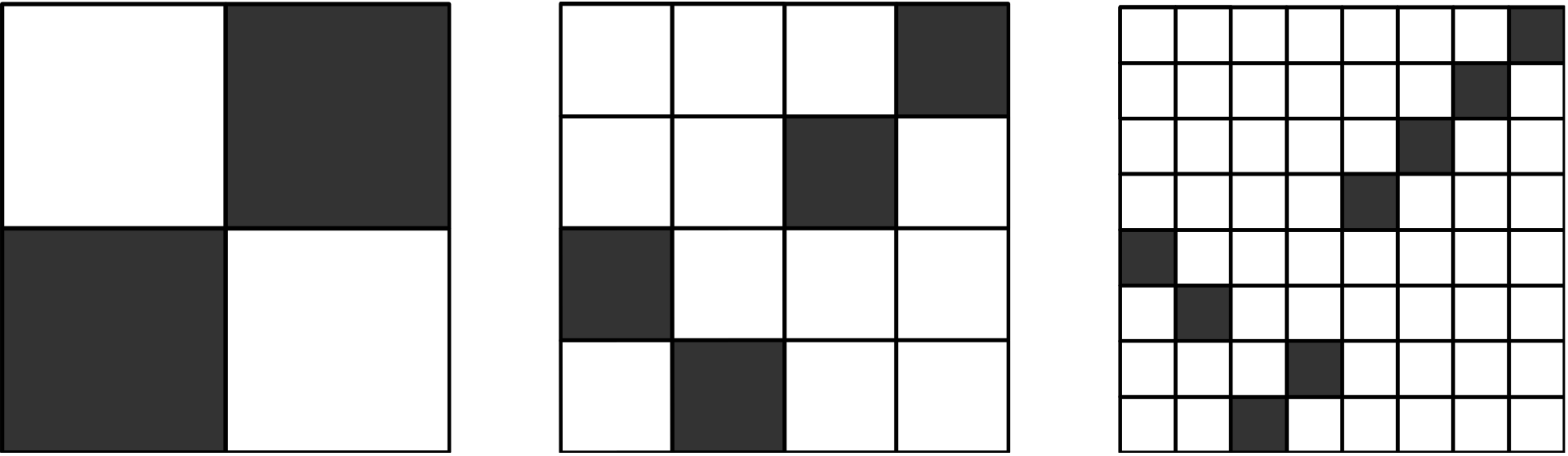}
\caption{The first three steps in the construction of F}
\label{fig4}
\end{figure}

If, on the other hand, we define another random process by changing $L$ and $R$ in the above construction to $\widetilde{L}$ and $\widetilde{Q}$, we end up with another random set, denoted by $F(2,2)$. For an example see Figure \ref{fig4}. 

Note that the construction of both random sets $E(2,2)$ and $F(2,2)$ are special cases of our random Cantor sets model which we described at the beginning of this section. Note that both of these constructions give rise to random sets as used in the proof of Theorem \ref{answer}.

In the end we are going to show that $E(2,2)$ and $F(2,2)$ are ''different''. Indeed for every element  $E$ of $E(2,2)$, we have $\pi_{y}(E)= \{0\} \times [0,1]$  and $ \mathcal{H}^{1}(\pi_{x}(E)) \leq 1/2$, where $\pi_x, \pi_y$ are projections onto $x$-axis, $y$-axis respectively. But  $\pi_{x}(F)= [0,1]\times\{0\}$ and  $\pi_{y}(F)= \{0\} \times [0,1]$ for all $F \in F(2,2)$.
\end{example}

\section{geometric part} \label{geometric part}

In this section, we present some geometric lemmas. The following results are  adapted  from \cite{ss}
to our setting. In \cite{ss}, Corollary \ref{use} is proved for lines. Here we give the detailed proof for  general affine subspaces of any dimension.

We are going to define the angle between a plane $W \in A(d,m)$ and a hyperplane $H \in A(d,d-1)$. We assume $W \in G(d,m)$ and $H \in G(d,d-1)$ first. We say that they have zero angle if $W \subset H$. Otherwise we have $H+W= \RR^d$ where
\begin{equation}
H+W:= \{h+w: h\in H, w \in W\}.
\end{equation}
Applying the basic dimension formula in linear algebra for $H$ and $W$, we have that $\dim(H\cap W)= m-1$. Thus for any $x \in H \cap W$, there is unique affine line $\ell_x \subset W, x \in \ell_x$, $\ell_x \perp( H\cap W).$  We choose an affine unit vector $e(x) \in \ell_x$ such that the root of  $e(x)$ is $x$.  Let 
\[
\theta(H,W):=\theta(H, \ell_x) 
\]
for some $x \in H\cap W$ (there is only one point in $H\cap W$ when $m=1$), where $\theta(H, \ell )$ is  the angle between the line $\ell$ and the plane $H$ defined in the usual manner.  Since $\ell_x$ and $\ell_y$ are parallel for any $x, y \in H \cap W$, the angle $\theta(H,W)$ doesn't depend on the choice of $x$. For the case that $W \in A(d,m)$ and $H \in A(d,d-1)$, there are unique subspaces  $W'\in G(d,m)$ and $H' \in G(d,d-1)$  parallel to $W$ and $H$, respectively. We define $\theta(H,W):=\theta(H', W')$.

Let $H_i= \{x \in \RR^d: x_i=0\}$ for $ 1 \leq i\leq d$. Define
\[
\Gamma_n(d,m) =\{W \in A(d,m): \min_{1\leq i \leq d}\theta(W,H_i) \geq r_n^d \}
\] and $\Gamma(d,m)= \cup_{n\in \mathbb{N}} \Gamma_n(d,m)$. In the following we  use $C(d)$ to represent  constants which  don't depend on $n$. We use $\# J$ to denote the cardinality of a set $J$. 

\begin{lemma}\label{geometric}
For any $n\in \NN$, there is $\Gamma'_n(d,m)\subset \Gamma(d,m)$ such that for any $W \in \Gamma_n(d,m)$, there exists $W' \in \Gamma'_n(d,m)$ with
\[
\mathcal{H}^m(W\cap Q)\leq \mathcal{H}^m( W' \cap Q ) + C(d)r_n^{d+m}
\]
for all $Q \in D_n$. Furthermore $\# \Gamma_n'(d,m) < r_n^{-C(d)}$. 
\end{lemma}
\begin{proof}
Define a metric among $\Gamma(d, m)$ by setting
\[
\rho (V,W) =\sup_{x\in [0,1]^{d}} | \pi_V(x)-\pi_W(x)|. 
\]

Let $\alpha=r_n^d$, and $\epsilon=r_n^{2d+1}$.  Let $\Gamma'_n(d,m)$ be an $\epsilon$-dense subset of $\Gamma(d,m)$ in the $\rho$-metric. There is such an $\Gamma'_n(d,m)$ with $\# \Gamma'_n(d,m) < \epsilon^{-C(d)}$.

Let $W \in \Gamma'_n(d,m)$, then we choose $W' \in A^1$ such that $\rho(W,W') \leq \epsilon$.
For any $r_n$-adic cube $Q$ of $D_n$ with $Q \cap W \neq \emptyset$, denote by $B(Q,W,\epsilon) = \{x\in Q\cap W: \dist(x, \partial Q) \leq \epsilon \}$ the  ''boundary part'' of $Q\cap W$ and by $I(Q, W, \epsilon) = (Q \cap W) \backslash B(Q,W,\epsilon)$ the ''interior part'' of  $Q\cap W$, see Figure \ref{fig1}.  We have that 
\begin{equation}\label{de}
Q \cap W = I(Q,W,2\epsilon) \cup B(Q,W,2\epsilon).
\end{equation}

\begin{figure}
\centering 
\includegraphics[width=0.6\textwidth]{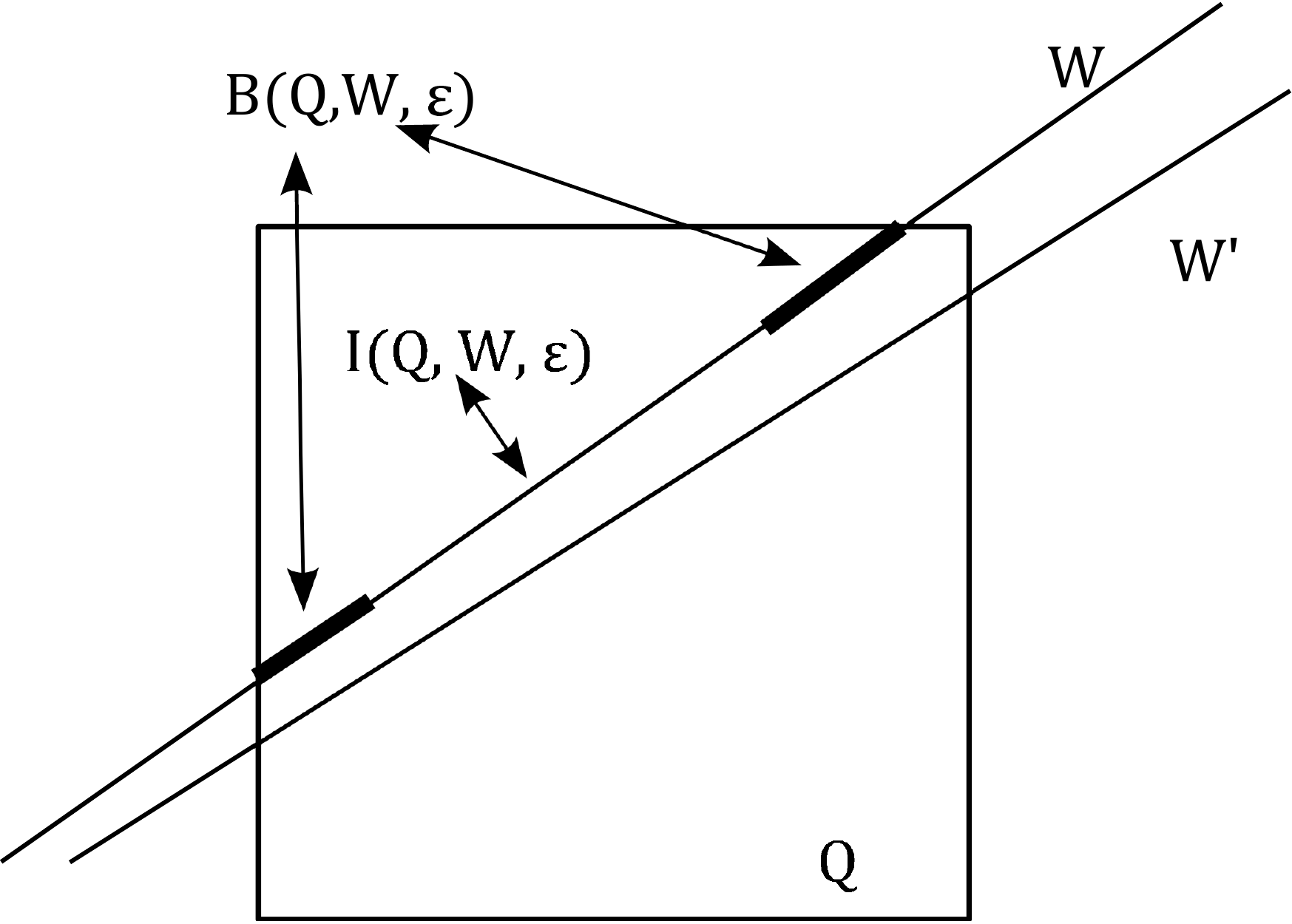}
\caption{Case one when $m=1$}
\label{fig1}
\end{figure}

Now we are going to show that $I(Q,W,2\epsilon) \subset \pi_W(W'\cap Q)$. For every $x \in I(Q,W,2\varepsilon)$ there is a unique $y \in W'$ such that $\pi_W (y)=x$. Since 
\[
\dist(x,y) = \dist(\pi_W(y), \pi_{W'}(y))\leq \varepsilon\]
 and $\dist(x, \partial(Q)) > 2\varepsilon$, we have $y \in B(x, \frac{3\varepsilon}{2} )\subset Q$. It follows that $I(Q,W,2\varepsilon)\subset \pi_{W}(W' \cap Q)$ and then 
\begin{equation}\label{inter}
\mathcal{H}^m (I(Q,W,2\epsilon)) \leq \mathcal{H}^m(Q \cap W'). 
\end{equation}

Now are going to show that $\mathcal{H}^m(B(Q,W, 2\epsilon ) \leq C(d)r_n^{d+m}$. For any $x \in W \cap \partial Q$, there exists at least one face of $Q$ which contains $x$. Then choose any such face and denote it by $F(x)$. Let $H(x)$ be the hyperplane which contains $F(x)$. Then there is a local  orthogonal basis at $x$, $\{e_1(x), e_2(x),...,e_m(x) \}$ of $W$, such that $e_m(x) \perp (W\cap H(x))$ and we denote by $(x_1, x_2, \cdots, x_m)$ the co-ordinates of $x$ with respect to this  basis. Then
\[
B(Q,W,2\epsilon) \subset \{ |x_m| \leq \frac{2\epsilon}{\sin\alpha} \text{ in the above local coordinates}\}.
\]
Note that once the face is fixed, $|x_m|$ does not depend on the choice of these local coordinates. Thus 
\[ 
\mathcal{H}^m(B(Q,W, 2\epsilon ) \leq (2\sqrt{m} r_n)^{m-1}\frac{2\epsilon}{\sin\alpha}. 
\]
There exists a constant $M\in \NN$, such that for $n\geq M$ imply  $\sin(r_n)>\frac{1}{2}r_n$. Thus we can choose a large constant $C(d)$ such that 
\begin{equation}\label{bound}
\mathcal{H}^m(B(Q,W, 2\epsilon ) \leq C(d)r_n^{d+m} 
\end{equation}
for all $n\in \NN.$ Applying the estimates \eqref{de}, \eqref{inter} and \eqref{bound}, we have 
\[
\mathcal{H}^m(W\cap Q) \leq \mathcal{H}^m( W'\cap Q) +C(d)r_n^{d+m}.
\] 
Thus we complete the proof.
\end{proof}

Let $m=d-k$ in Lemma \ref{geometric} and recall that the number of $r_n^{-1}$ -adic subcubes of $E_n$ is at most $r_n^{-d}$. Let $ \Gamma_n:=\Gamma_n(d,d-k)$  and $ \Gamma:=\Gamma(d,d-k)$. Let $A, B$ be two subset of  $\RR^d$. Define
\[
|A\cap B|:= \mathcal{H}^{d-k}(A\cap B).\] We have the following easy corollary. 

\begin{corollary}\label{use}
For any $n\in \NN$, there is  $\Gamma'_n \subset \Gamma$ such that for  any $W\in \Gamma_n$, there exists $W' \in \Gamma'_n$ with  
\[
|W\cap E_n| \leq |W' \cap E_n|+ C(d)r_n^{d-k}
\]
for any realization $E_n$. Further more $\# \Gamma_n'\leq r_n^{-C(d)}.$
\end{corollary}
\begin{proof}
By Lemma \ref{geometric}, we have that for any $W\in \Gamma_n$, there exist $W' \in \Gamma'_n$ such that 
\begin{equation}\label{cube}
|W\cap Q| \leq |W'\cap Q|+C(d)r_n^{2d-k}
\end{equation}
for each $Q \in D_n$ and    $ \# \Gamma'_n< r_n^{-C(d)}$. 
For any realization $E_n$, we sum the two sides of \eqref{cube} over 
$Q\in D_n$ such that $Q \subset E_n$: 
\begin{equation}
\sum_{Q \subset E_n}|W \cap Q | \leq  \sum_{Q \subset E_n}|W' \cap Q |+C(d)r_n^{d-k}.
\end{equation}
By the definition of $|W\cap E_n|$, we arrive at the required estimate. 
\end{proof}

For a strip $S\in S (d, d-k)$, denote
\[
Z(S,n)=\# \{Q ~~is~~ an ~~r_n~~adic~~cube \mid Q\cap E_n\cap S\neq \emptyset\}.
\]
For  later use in Corollary \ref{limsup2}, we state the following easy fact as a Lemma. 

\begin{lemma}\label{length and area}
If $ |W\cap E_n|\leq h$ for all  $W \in \Gamma_n$, then for any strip $S \in S(d,d-k)$ with width $0<w(S)\leq r_n$, we have
\[
Z(S,n) \leq C(d) r_n^{k-d} h.
\]
\end{lemma}
\begin{proof}
We assume  $S(V) \in S(d,d-k)$ with  $V \in\Gamma_n$ first. Let $V^{\perp} \in G(d,k)$ be the orthogonal complement of $V$ and $z:= V^{\perp} \cap V$. Let $B_{V^{\perp}}( y,r)$ be the ball of $V^{\perp}$ with center y and radius $r$. Let $t_n:= r_n /2 + \sqrt{d}r_n$. Since 
\[
\{Q \text{ is an $r_n$ adic cube} : Q \cap E_n \cap S \neq \emptyset \} \subset V(t_n)
\] 
where $V(t_n)$ is the $t_n$ neighborhood of $V$ in $\mathbb{R}^d$,  we have
\begin{equation}\label{l1}
Z(S,n) r_n^{d}\leq \mathcal{H}^{d}(E_n \cap V(t_n) ).
\end{equation} 
Using Fubini's theorem and the condition $|W\cap E_n|\leq h$  for all  $W \in \Gamma_n,$
we obtain that
\begin{equation}\label{l2}
\begin{aligned}
\mathcal{H}^{d}(E_n \cap V(t_n) ) = \int_{B_{V^{\perp}}(z, t_n )} \mathcal{H}^{d-k}  (E_n \cap P^{-1}_{V^{\perp}}(x)) d\mathcal{H}^{k}(x) \leq  (t_n) ^{k}h.
\end{aligned}
\end{equation}
By \eqref{l1} and \eqref{l2} we have $Z(S,n)r_n^d \leq (t_n)^k h$. Let $C_1(d):=(1+2\sqrt{d})^k$. We get that 
$Z(S,n) \leq C_1(d)r_n^{k-d} h$. 

For a strip $S(V)$ with $V \in \Gamma_n ^{c}$ ( $\Gamma_n ^{c}$ is the complement of $\Gamma_n$ that is  $A(d,d-k) \backslash \Gamma_n$) and $w(S)\leq r_n$, there is a strip 
\[\widetilde{S}=\{x\in \RR^d: \dist(x, \widetilde{V}) \leq 5 r_n\}\]
 and $S \subset  \widetilde{S}$. Thus
\[Z(S,n)\leq Z(\widetilde{S},n)\leq  10 C_1(d)r_n^{k-d}h.\] 
Let $C(d)=10C_1$. Thus  the proof is completed.
\end{proof}

Note that the constant $C(d)$ may be different in different places of this section. For the convenience in what follows we fix a constant $C(d)$ such that the statements of all the Lemmas and Corollaries hold with this constant.

\section{probabilistic part}\label{probability part}

We use a similar method as in \cite{ckls} to estimate the intersections of our sets with affine planes. The random Cantor sets studied in \cite{ckls} are different from the ones considered here. We choose interior disjoint closed subcubes at every step of our constructions, while in \cite{ckls} overlaps are allowed. Since we assume that $M_n$ are uniformly bounded, the proof here will be simpler than that of \cite{ckls}. On the other hand, we give here the detailed proof for general $d$ and $m$ while in the main part of \cite{ckls}, it is assumed that $d=2, m=1.$

We fix a number $t<s\leq k$ and let $0<5\epsilon\leq s-t$. Recall that by \eqref{dimension}, there exists $n_0 \in \NN$ such that
\begin{equation}\label{dim}
r_m^{-t-4\varepsilon}\leq r_m^{-s+\varepsilon} \leq P_{m}\leq r_m^{-s-\varepsilon}
\end{equation}
holds for all $m \geq n_0, m \in \NN$. For this $n_0$, there is a constant  $R_0$ such that 
\begin{equation}\label{choice}
|W\cap E_{n_0}|\leq R_0P_{n_0} r_{n_0}^{t+d-k}, 
\end{equation}
holds for all $W \in A(d,d-k)$ and any realization $E_{n_0}$. Let  $W \in \Gamma$, $n\in\mathbb N$. Define 
\[
Y_n^{W}= (P_nr_n^d)^{-1}| W\cap E_n |. 
\]

Denote by  $\PP(\cdot \big| A)$
 the conditional probability conditioned on the event $A$.


\begin{lemma} \label{conditionalprobabilityestimate}
Let  $n > n_0, n \in \NN$ and  $W\in \Gamma$. Then for any positive $\lambda$ and $\lambda_0$ with $\lambda (2\sqrt{d} r_{n-1})^{d-k} (P_nr_n^d)^{-1}\leq \lambda_0\leq 1$, we have
\begin{equation}
\EE\left(e^{\lambda Y_n^{W}} \Big| E_{n-1}\right)\leq e^{(1+\lambda_0)\lambda Y^{W}_{n-1}}.
\end{equation}
\end{lemma}
\begin{proof}
Let $Q_{1}, Q_{2}, \cdots , Q_{K}$  be the cubes in $E_{n-1}$ hitting $W$. For each  $1\leq i\leq K$, consider the random variable
\begin{equation}
\begin{aligned}
X_{i}= (P_nr_n^{d})^{-1} \mathcal{H}^{d-k} \left(W\cap E_{n} \cap Q_i  \right). 
\end{aligned}
\end{equation}
Thus we have $Y_n^{W}= \sum_{i=1}^{K} X_{i}$.  For each $1\leq i\leq K$, we have that  
\begin{equation}\label{key1}
\EE( X_i \Big| E_{n-1} ) = (P_{n-1}r_{n-1}^{d})^{-1} |W\cap Q_i|.
\end{equation}
Conditional on $E_{n-1}$, recall that the cubes forming $E_{n}$ are chosen   independently inside each $Q_i, 1\leq i\leq K$.  Thus the random variables $X_{i}, 1\leq i \leq K$ are independent. And so $ e^{\lambda X_i}$ are also independent. This gives
\begin{equation}\label{independent}
\mathbb{E} \left( e^{\lambda Y_n^{W} }\Big| E_{n-1}\right)= \prod^K_{i=1}\mathbb{E} \left( e ^{\lambda X_{i}}\Big| E_{n-1}\right).
\end{equation}
For all $ \vert x \vert \leq \rho\leq 1$, we use the fact
$e^{x} \leq 1+(1+\rho)x$ and
\[
\lambda X_i\leq \lambda (2\sqrt{d} r_{n-1})^{d-k} (P_nr_n^d)^{-1} \leq \lambda_0\leq 1
\]
 for $1\leq i \leq K$, to obtain
\begin{equation}\label{11}
e^{\lambda X_i } \leq 1+(1+\lambda_0)\lambda X_i.
\end{equation}
Thus by \eqref{key1} and the trivial inequality $  1+x \leq e^x$, we have 
\begin{equation}\label{one estimate}
\begin{aligned}
\EE \left( e ^{\lambda X_{i}} \Big| E_{n-1}\right) \leq \exp \left(\left(1+\lambda_0 \right)\lambda (P_{n-1}r_{n-1}^{d})^{-1}| W\cap Q_i | \right).
\end{aligned}
\end{equation}
Combing this with \eqref{independent} and the definition of $Y^{W}_{n-1}$, we finish the proof. 
\end{proof}
Let $R>2R_0 C(n_0)$ be a constant where 
\begin{equation}\label{eq:cc}
C(n_0):= \left(2\sqrt{d}M \right)^{d-k} \prod^{\infty}_
{i=n_0+1} \left(1+r_i^{\varepsilon}\right). 
\end{equation}
By applying Lemma \ref{conditionalprobabilityestimate} and the total expectation formula, we have the following estimate. 

\begin{lemma} \label{probabilityestimate}
For any  $n > n_0, n \in \NN$ and  $W\in \Gamma$, we have the bound
\begin{equation}
\PP\left(Y_n^W> Rr_{n}^{t-k}  \right)\leq \exp\left(-r_{n}^{-\epsilon}\right).
\end{equation}
\end{lemma}
\begin{proof}
Let $\lambda=C(n_0)^{-1}
P_nr_n^{k+3\varepsilon}$. We apply Markov's inequality to the random variable $e^{\lambda Y_n^{W}}$. This gives
\begin{equation}\label{eq:markov}
 \PP\left(Y_n^{W}> Rr_{n}^{t-k}  \right) \leq e^{-\lambda Rr_{n}^{t-k}} \EE(e^{\lambda Y_n^{W}}).
\end{equation}
Now we are going to estimate $\EE(e^{\lambda Y_n^{W}})$. 
By the choice of $\lambda$, we have 
\[
\lambda (2\sqrt{d}r_{n-1})^{d-k}(P_nr_n^{d})^{-1}\leq r_n^{\varepsilon}.
\]
Applying Lemma \ref{conditionalprobabilityestimate} we have
\begin{equation}\label{a}
 \EE \left(e^{\lambda Y_n^{W}}  \Big|E_{n-1}\right)
 \leq e^{(1+r_n^{\varepsilon})\lambda Y_{n-1}^{W}}. 
\end{equation}
The total expectation formula and estimate \eqref{a} imply 
\begin{equation}\label{total}
 \EE \left(e^{\lambda Y_n^{W}}\right)= \EE\left( \EE\left(e^{\lambda Y_n^{W}} \Big| E_{n-1}\right)\right)\leq \EE\left(e^{(1+r_n^{\varepsilon})\lambda Y_{n-1}^{W}}\right).
\end{equation}
By the choice of $\lambda, n\geq n_0$ and estimate \eqref{dim}, we see that
\[
\frac{ \lambda (2\sqrt{d}r_{j-1})^{d-k}   \prod^{n}_{i=j+1}(1+r_i^{\varepsilon})  }{ P_jr_j^{d}} \leq \frac{P_nr_n^{k+3\varepsilon} r_j^{d-k}}{P_jr_j^{d}}\leq r_j^{\varepsilon},
\]
holds for all $n_0<j<n$. Applying \eqref{total} inductively, we have
\begin{equation}
\begin{aligned}\label{eq:c}
\EE \left(e^{\lambda Y_n^{W}}\right)
 \leq \EE\left(e^{\lambda Y_{n_0}^{W}\prod^{n}_{i=n_0+1}(1+r_i^{\varepsilon})}\right)
\leq \exp \left({P_n r_n^
{k+3\varepsilon}R_0r_{n_0}^{t-k}}\right).
\end{aligned}
\end{equation}
The last inequality holds by our choice of $\lambda$ and the condition that  
\begin{equation}
|W\cap E_{n_0}|\leq R_0P_{n_0} r_{n_0}^{t+d-k} 
\end{equation}
holds for all $W \in \Gamma$. Combining the estimates \eqref{eq:c}  and \eqref{eq:markov}, we have 
\begin{equation}
\begin{aligned}
 \PP\left(Y_n^{W}>  Rr_{n}^{t-k}\right) &\leq \exp\left(-\lambda Rr_{n}^{t-k} + {P_n r_n^
{k+3\varepsilon}R_0r_{n_0}^{t-k}}\right) \\
&= \exp(-P_nr_n^{k+3\varepsilon}r_n^{t-k}(C(n_0)^{-1}R-R_0r_{n_0}^{t-k}r_n^{k-t}))\\
& \leq \exp({-r_n^{-\varepsilon}}).
\end{aligned}
\end{equation}
The last inequality holds since $P_nr_n^t\geq r_n^{-4\varepsilon}$ and $R>2R_0 C(n_0)$ (we also ask that $R_0 >1$). 
\end{proof}

Let $n\in \NN$ and $W \in \Gamma$. Denoted by $G_n(W)$ the (good) event 
 \[
|W\cap E_n|\leq RP_{n} r_{n}^{t+d-k} + C(d)r_n^{d-k}.
\]
Let $G_n$ be the event that $G_n(W)$ holds for all $W\in \Gamma_n$. Applying Corollary \ref{use} and Lemma \ref{probabilityestimate}, we have the following result.

\begin{corollary}\label{limsup1}
We have $\PP(\cup^{\infty}_{k=1}
\cap^{\infty}_{n=k} G_n)=1.$
\end{corollary}
\begin{proof}
By Corollary \ref{use} we know that if  the estimate 
\[
| W^{'}\cap E_{n}^{\omega}|\leq RP_{n}r_{n}^{t+1}
\]
holds for all $W^{'}\in \Gamma_{n}'$, then the estimate
\[
| W\cap E_{n}^{\omega}|\leq RP_{n}r_{n}^{t+1}+C(d)r_n^{d-k}.
\]
holds for any $W \in \Gamma_n$. Thus $\omega \in G_n$. Let $n\geq N_0, n \in\NN$. Then by the above argument we have 
\begin{align*}
\PP ( G_{n}^c )
\leq \PP(|W&\cap E_{n}^{\omega}|> P_{n}R_{n}r_{n}^{t+d-k}\textrm {  
 for some }  W'\in \Gamma^{'}_{n} )\\
 &\leq r_{n}^{-C(d)}\exp (-r_{n}^{-\epsilon} ),
\end{align*}
where $G_{n}^{c}$ is the complement of $G_{n}$. The last inequality holds by Lemma \ref{probabilityestimate}. Note that there are at most $r_{n}^{-C(d)}$ elements in $\Gamma^{'}_{n}$. 

Since the series $\sum^{\infty}_{n=1} r_{n}^{-C(d)}\exp (-r_{n}^{-\epsilon} )$ converges, the Borel-Cantelli lemma implies $\PP( \cap^{\infty}_{k=1}
\cup^{\infty}_{n=k}G_{n}^{c})=0$. 
\end{proof}

Now are going to estimate the distribution of  the natural measure $\mu= \mu^{\omega}$ on strips. 

\begin{corollary}\label{limsup2} 
Let $\omega \in  \cup^{\infty}_{k=1}
\cap^{\infty}_{n=k} G_n$. Then
\begin{equation}\label{part}
\sup_{S\in S(d,d-k)} \frac{\mu^{\omega}(E^{\omega} \cap S) }{w(S)^t}<\infty.
\end{equation}
\end{corollary}
\begin{proof}
Since $\omega \in  \cup^{\infty}_{k=1}
\cap^{\infty}_{n=k} G_n$, there exits $n_{\omega}$ such that $E^\omega \in G_{n}$ for all $n \geq n_\omega$. It means the estimate 
\[
|W\cap E_n|\leq RP_{n} r_{n}^{t+d-k} + C(d)r_n^{d-k}\leq 2RP_{n} r_{n}^{t+d-k}
\]
holds for all $n\geq n_\omega$ and all $W \in \Gamma_n$.  The last inequality holds by choosing large $R$. 

Let $S(W) \in S(d,d-k)$ with width $w(S)$. We assume  $w(S) \leq r_{n_\omega}$ first. There exists $n\geq n_\omega$ such that $r_{n+1}< w(S) \leq r_n$. By Lemma \ref{length and area}, we have $Z(S, n) \leq 2C(d)RP_nr_n^t$. Thus 
\[
\mu(S)\leq 2C(d)R  r_n^t \leq 2C(d)R M^t w(S)^{t}.
\]

In the case that $w(S)>r_{n_\omega},$ it's trivial to see that  
\begin{equation}
\sup_{ \substack{S\in S(d,d-k)\\ w(S)>r_{n_\omega} } }\frac{\mu(E \cap S) }{w(S)^t}\leq r_{n_\omega}^{-1}, 
\end{equation}
since  $\mu$ is a probability measure. Thus we complete the proof.
\end{proof}

Notice that all the above claims hold for any $t<s$. Letting $t\rightarrow s$ through a countable sequence gives the proof of Lemma \ref{main lemma}:
\begin{proof}[Proof of Lemma \ref{main lemma}.]
Let $0 <t_1<t_2<\cdots$ such that $t_k\nearrow s$.   For every $t_k$, we denote by $\Omega_k$ the event  
\begin{equation}\label{t_k}
\sup_{S\in S(d,d-k)} \frac{\mu(E \cap S) }{w(S)^{t_k}}<\infty.
\end{equation}
By Corollary \ref{limsup1} and Corollary \ref{limsup2}, we have $\PP(\Omega_k)=1$. So $\PP(\cap_{k=1}^{\infty}
\Omega_k)=1$ as well. Let $\omega \in \cap_{k=1}^{\infty}
\Omega_k $, then $\mu^\omega$ satisfies \eqref{t_k} for every $t_k$.

For any $t<s$, there is $t_k$, such that $t<t_k<s$. We have $w(S)^{t}\geq w(S)^{^{t_k}}$ when $w(S)\leq 1$. Thus 
\begin{equation}\label{last}
\sup_{ \substack{S\in S(d,d-k)\\ w(S)\leq 1} } \frac{\mu(E \cap S) }{w(S)^t}\leq \sup_{ \substack{S\in S(d,d-k)\\ w(S)\leq 1} } \frac{\mu(E \cap S) }{w(S)^{t_k}} <\infty.
\end{equation}
Again since  $\mu$ is a probability measure we have 
\begin{equation}
\sup_{ \substack{S\in S(d,d-k)\\ w(S)>1} } \frac{\mu(E \cap S) }{w(S)^t}\leq1.
\end{equation}
Combing this with the estimate \eqref{last}, the claim follows.
\end{proof}

\end{document}